\documentclass{amsart}
\usepackage{latexsym, amscd, amsfonts, eucal, mathrsfs, amsmath, amssymb, amsthm, stmaryrd}
\usepackage[all,arc]{xy}
\usepackage{tikz}
\usepackage{tikz-3dplot}
\usetikzlibrary{matrix}

\def\p{\prime}

\def\bl{\backslash}

\def\isom{\simeq}
\def\inj{\hookrightarrow}
\def\sse{\subseteq}

\def\es{\emptyset}

\def\hom{\mathrm{Hom}}

\def\cC{\mathcal C}

\def\Lm{\Lambda}

\def\al{\alpha}

\def\ot{\otimes}
\def\t{\times}
\def\mto{\mapsto}

\def\xrw{\xrightarrow}

\def\id{\mathrm{id}}

\def\De{\Delta}

\def\cO{\mathcal O}

\def\cP{\mathcal{P}}

\def\cC{\mathcal{C}}

\def\cI{\mathcal{I}}
\def\cJ{\mathcal{J}}

\def\der{\partial}

\def\bscup{\bigsqcup}

\def\Ac{\mathbf{Ac}}
\def\sSet{\mathbf{sSet}}
\def\Pos{\mathbf{Pos}}
\def\Cat{\mathbf{Cat}}

\def\Set{\mathbf{Set}}

\def\Sd{\mathrm{Sd}}
\def\Ex{\mathrm{Ex}}
\def\lrs{\leftrightarrows}
\def\rls{\rightleftarrows}

\def\sV{\mathscr{V}}
\def\Lan{\mathrm{Lan}}
\def\opp{\mathrm{op}}
\def\colim{\mathrm{colim}}

\def\Id{\mathrm{Id}}
\def\ob{\mathrm{ob}}
\def\cM{\mathcal{M}}
\def\cN{\mathcal{N}}

\def\ihom{\mathcal{H}{om}}
\def\rt{\rtimes}

\def\mor{\mathrm{mor}}

\newtheorem{thm}{Theorem}[section]
\newtheorem{lemma}[thm]{Lemma}
\newtheorem{prop}[thm]{Proposition}
\newtheorem{coro}[thm]{Corollary}

\theoremstyle{definition}

\newtheorem{defn}[thm]{Definition}

\newtheorem{rmk}[thm]{Remark}

\begin{document}
  \title{Generalized equivariant model structures on $\Cat^I$}
  \author{Yuzhou Gu}
  \date{}
  \begin{abstract}
  Let $I$ be a small category, $\cC$ be the category $\Cat$, $\Ac$ or $\Pos$ of small categories, acyclic categories, or posets, respectively. Let $\cO$ be a locally small class of objects in $\Set^I$ such that $\colim_I O=*$ for every $O\in \cO$. 
  We prove that $\cC^I$ admits the $\cO$-equivariant model structure in the sense of Farjoun, and that it is Quillen equivalent to the $\cO$-equivariant model structure on $\sSet^I$. This generalizes previous results of Bohmann-Mazur-Osorno-Ozornova-Ponto-Yarnall and of May-Stephan-Zakharevich when $I=G$ is a discrete group and $\cO$ is the set of orbits of $G$. 
  \end{abstract}
  \maketitle
  
  \tableofcontents
  
  \section{Introduction}
  In \cite{Tho80}, Thomason defined a model structure on $\Cat$, the category of small categories, which is Quillen equivalent to the standard model structure on $\sSet$. The Thomason model structure has been shown to transfer to other categories. 
  Raptis \cite{Rap10} showed that $\Pos$, the category of posets, admits a model structure that is Quillen equivalent to the Thomason model structure on $\Cat$. 
  Recently Bruckner \cite{Bru15} showed that $\Ac$, the category of small acyclic categories, admits a model structure that is Quillen equivalent to the Thomason model structure on $\Cat$.
  The following diagram shows the relevant Quillen equivalences. 
    \[
   \xymatrix{\sSet\ar@<.5ex>[r]^{c\Sd^2} 
   &\Cat\ar@<.5ex>[l]^{\Ex^2 N} 
   \ar@<.5ex>[r]^{p_A}
   & \Ac \ar@<.5ex>[l]^{i_A}
   \ar@<.5ex>[r]^{p_P}
   & \Pos \ar@<.5ex>[l]^{i_P}
   }
  \]
  In the diagram, $N:\Cat\to \sSet$ is the nerve functor, $c$ is its left adjoint, $\Sd:\sSet\to \sSet$ is the barycentric subdivision, and $\Ex$ is its right adjoint. 
  The functors $i_A$ and $i_P$ are the natural inclusions, and $p_A$ and $p_P$ respectively are their left adjoints. 
  
  Let $G$ be a discrete group. The above Quillen equivalences have been generalized to $G$-spaces.
  In \cite{BMOOPY13} Bohmann-Mazur-Osorno-Ozornova-Ponto-Yarnall showed that the Quillen equivalence between $\sSet$ and $\Cat$ can be lifted to a Quillen equivalence between $\sSet^G$ and $\Cat^G$, each equipped with the fixed point model structure. 
  May-Stephan-Zakharevich \cite{MSZ16} further showed that the Quillen equivalence between $\Cat^G$ and $\Pos^G$ can also be lifted. 
  The following diagram shows the relevant Quillen equivalences. 
   \[
   \xymatrix{
   \sSet^G\ar@<.5ex>[r]^{c\Sd^2}
   &\Cat^G\ar@<.5ex>[l]^{\Ex^2 N} 
   \ar@<.5ex>[r]^{p_P \circ p_A}
   & \Pos^G \ar@<.5ex>[l]^{i_A \circ i_P}
   }
  \]
  
  Our main result is a generalization of results of \cite{BMOOPY13} and of \cite{MSZ16} to diagram categories indexed by arbitrary small categories.
  The $G$-fixed point model structures are replaced with the $\cO$-equivariant model structures, where a morphism $X\to Y$ is a weak equivalence (resp. fibration) if and only if for every $O\in \cO$, the induced map $\hom(O,X)\to \hom(O,Y)$ is a weak equivalence (resp. fibration) in $\sSet$. (The full definition is given in section \ref{SecOEquiv}). We also need to specify $\sSet$-enriched structures on $\Cat^I$, $\Ac^I$ and $\Pos^I$, which are different from the usual $\sSet$-enriched structures and are discussed later. 
  
  \begin{thm} \label{MainThm}
  Let $I$ be a small category, $\cC$ be the category $\Cat$, $\Ac$ or $\Pos$. Let $\cO$ be a locally small class of diagrams in $\Set^I$ such that $\colim_I O=*$ for every $O\in \cO$. 
  Then $\cC^I$ admits the $\cO$-equivariant model structure, and this model structure is Quillen equivalent to $\sSet^I$ equipped with the $\cO$-equivariant structure.
  \end{thm}
  When $I=G$ is a discrete group and $\cO=\cO_G$ is the category of $G$-orbits, the theorem reduces to \cite{BMOOPY13}, Theorem A and B, and \cite{MSZ16}, Theorem 1.1 and 1.2.
  
  The proofs of \cite{BMOOPY13} and \cite{MSZ16} are based on a theorem of Stephan \cite{Ste13}, which says that for suitable categories $\cC$, the category $\cC^G$ admits the fixed point model structure, and this model structure is Quillen equivalent to $\cC^{\cO_G^{\opp}}$ equipped with the projective model structure.
  This can be seen as a generalization of Elmendorf's theorem \cite{Elm83}, which says that $\sSet^G$ equipped with the fixed point model structure is Quillen equivalent to $\sSet^{\cO_G^\opp}$ equipped with the projective model structure. 
  
  In our proof of Theorem \ref{MainThm} we use a generalization of Elmendorf's theorem in another direction. 
  Dwyer and Kan \cite{DK84} proved that, for a bicomplete $\sSet$-enriched cateogry $\cM$ and a small full subcategory of orbits $\cO\sse \cM$ satisfying certain axioms, $\cM$ admits the $\cO$-equivariant model structure, and this model structure is Quillen equivalent to $\sSet^{\cO^{\opp}}$ with the projective model structure.
  
  Farjoun \cite{Far87} generalized the $\cO$-equivariant model structure to cases where $\cO$ can be a proper class rather than a set. 
  Farjoun also applied this theory to $\cM=\sSet^I$, where $I$ is a small category. 
  In this case, Farjoun showed that $\cO_I$, the class of all diagrams whose colimits over $I$ are points, is a collection of orbits. 
  
  However, it is not easy to prove an Elmendorf's theorem when $\cO$ is a proper class, because in this case, $\sSet^{\cO^\opp}$, ``the category of functors from $\cO^\opp$ to $\sSet$'', is not well-defined. This issue is partially resolved by Chorny and Dwyer \cite{CD09}, and eventually resolved by 
  Chorny \cite{Ch14}, by replacing $\sSet^{\cO^\opp}$ with $\cP(\cM)$, the category of small functors from $\cM^\opp$ to $\sSet$, equipped with the $\cO$-relative model structure (recalled in section \ref{SecChorny}). 
  Chorny's theorem says that $\cM$ equipped with the $\cO$-equivariant model structure is Quillen equivalent to $\cP(\cM)$ equipped with the $\cO$-relative model structure.
  Chorny's theorem is an essential ingredient of our proof of Theorem \ref{MainThm}. 
  
  \section{Acknowledgements}
  This work is supported in part by MIT Undergraduate Research Opportunities Program (UROP). The author is very grateful to Dr. Emanuele Dotto, for proposing the problem, supervising the author, and giving very helpful advice on various aspects. 
  
  \section{Preliminaries}
  In this section we review the necessary definitions and results. 
  
  \subsection{The Thomason model structure}
  Thomason \cite{Tho80} defined a model structure on $\Cat$ in which a morphism $f$ in $\Cat$ is a weak equivalence (resp. fibration) if and only if $\Ex^2 N f$ is a weak equivalence (resp. fibration) in $\sSet$ equipped with the standard model structure, where $N:\Cat\to \sSet$ is the nerve functor, and $\Ex:\sSet\to \sSet$ is the right adjoint of the barycentric subdivision functor $\Sd:\sSet\to \sSet$.
  
  Let $\cI_{\sSet}$ be the set of generating cofibrations $\der \De[n]\to \De[n]$, and $\cJ_{\sSet}$ be the set of generating trivial cofibrations $\Lm^k[n]\to \De[n]$ for the standard model structure on $\sSet$. 
  Then $\cI_{\Cat}=c\Sd^2 \cI_{\sSet}$ is a set of generating cofibrations and $\cJ_{\Cat}=c\Sd^2 \cJ_{\sSet}$ is a set of generating trivial cofibrations for the Thomason model structure on $\Cat$, where $c:\sSet\to \Cat$ is the left adjiont to $N$ and $\Sd:\sSet\to \sSet$ is the barycentric subdivision functor, which is left adjoint to $\Ex$. 
  
  The notion of Dwyer maps is important for the Thomason model structure. 
  \begin{defn}[\cite{Tho80}, Definition 4.1]
  Let $i:A\to B$ be a monomorphism in $\Cat$.
  
  We say $i$ is a \emph{sieve} if for every $a\in \ob(A)$ and morphism $f:b\to i(a)$ in $B$, there exists morphism $f^\p:b^\p\to a$ in $A$ such that $i(f^\p)=f$. 
  
  We say $i$ is a \emph{cosieve} if for every $a\in \ob(A)$ and morphism $g:i(a)\to b$ in $B$, there exists morphism $f^\p:a\to b^\p$ in $A$ such that $i(f^\p)=f$. 
  
  We say $i$ is a \emph{Dwyer map} if it is a sieve and factorizes as $A\xrw{f} W\xrw{j} B$ such that $f$ is a monomorphism, $j:W\to B$ is a cosieve and there is a right adjoint $r:W\to A$ to $f$. 
  \end{defn}
  
  The category $\Ac$ of small acyclic categories and the category $\Pos$ of posets are reflective subcategories of $\Cat$. Both of them admits the Thomason model structure, by Raptis \cite{Rap10} for $\Pos$ and Bruckner \cite{Bru15} for $\Ac$. 
  In the Thomason model structures on $\Ac$ and  $\Pos$, a morphism is a weak equivalence (resp. fibration) if and only if it is a weak equivalence (resp. fibration) as a morphism in $\Cat$. 
  All morphisms in $\cI_\Cat$ and $\cJ_\Cat$ are Dwyer maps between posets. 
  For $\cI=\Ac$ and $\Pos$, write $\cI_\cC=\cI_\Cat$ and $\cJ_\cC=\cJ_\Cat$.
  Then $\cI_\cC$ is the set of generating cofibrations and $\cJ_\cC$ is the set of generating trivial cofibrations for the Thomason model structure on $\cC$. 
  
  \subsection{The $\cO$-equivariant model structure}\label{SecOEquiv}
  Let $\cM$ be a bicomplete $\sSet$-enriched category. Here ``bicomplete'' means that the underlying category of $\cM$ is bicomplete, and $\cM$ is powered and copowered (i.e. cotensored and tensored) over $\sSet$. Let $\ot$ denote the copower structure on $\cM$. Let $\hom(-,-)$ denote the $\sSet$-enriched hom. Let $\cO$ be a class of objects of $\cM$. 
  
  \begin{defn}
  We say $\cM$ admits the \emph{$\cO$-equivariant model structure} if there is a model structure on $\cM$ such that a morphism $X\to Y$ in $\cM$ is a weak equivalence (resp. fibration) if and only if for every $O\in \cO$, the induced map $\hom(O,X)\to \hom(O,Y)$ is a weak equivalence (resp. fibration) in $\sSet$ with the standard model structure. 
  \end{defn}

  \begin{defn}[\cite{Far87}, Definition 1.1]
  A class $\cO$ of objects of $\cM$ is called \emph{locally small} if for every object $M\in \cM$ there exists a set of objects $\cO^\p\sse \cO$ such that every morphism from an object in $\cO$ to $M$ factors through an object in $\cO^\p$.
  \end{defn}
  \begin{rmk}
  In particular, if $\cO$ is a set, then $\cO$ is locally small. 
  \end{rmk}
  
  Dwyer and Kan \cite{DK84} defined a collection of orbits when $\cO$ is a set. 
  Farjoun \cite{Far87} generalized the definition to the case when $\cO$ can be a proper class. 
  \begin{defn}[\cite{DK84}, \cite{Far87}] \label{OrbitsDefn}
  A locally small class $\cO$ of objects in $\cM$ is called \emph{a collection of orbits}, if the following axioms hold. 
  \begin{enumerate}
  \item[Q1.] Let \[
\xymatrix{
O \ot K \ar[r] \ar[d]  &  X_{a} \ar[d]\\
O \ot L \ar[r] &  X_{a+1}\\
}
\] be a pushout in $\cM$ where $O\in \cO$ and  $K\to L\in \cI_{\sSet}$. 
Then for $O^\p\in \cO$, the diagram \[
\xymatrix{
\hom(O^\p,O \ot K) \ar[r] \ar[d]  &  \hom(O^\p,X_{a}) \ar[d]\\
\hom(O^\p,O \ot L) \ar[r] &  \hom(O^\p,X_{a+1})\\
}
\] is a homotopy pushout in $\sSet$. 

  \item[Q2.] Let $O\in \cO$ and $X_0\to \cdots \to X_a\to X_{a+1}\to \cdots$ be a continuous transfinite sequence in $\cM$ where each map $X_a\to X_{a+1}$ is as in Q1. Then 
  the natural map $$\colim_a \hom(O, X_a)\to \hom(O,\colim_a X_a)$$
  is a weak equivalence in $\sSet$. 
  
  \item[Q3.] There exists a limit ordinal $c$ such that if the sequence in Q2 is indexed by ordinals $<c$, then the natural map is an isomorphism. 
  \end{enumerate}
  \end{defn}
  
  The following proposition is proved by Dwyer and  Kan \cite{DK84} when $\cO$ is a set, and by Farjoun \cite{Far87} when $\cO$ can be a proper class. 
  
  \begin{prop}[\cite{Far87}, Proposition 1.3] \label{EquivModelStructure}
  Let $\cO\sse \cM$ be a collection of orbits.
  Then $\cM$ admits an $\cO$-equivariant model structure.
  \end{prop}
  
  \subsection{Chorny's theorem}\label{SecChorny}
  Let $\cM$ be a bicomplete $\sSet$-enriched category and $\cO$ be a collection of orbits in $\cM$. 
  Chorny's theorem compares the $\cO$-equivariant model structure on $\cM$ with the $\cO$-relative model structure on $\cP(\cM)$, and is used in our proof of Theorem \ref{MainThm}. 
  
  Recall the definition of $\cP(\cM)$. 
  \begin{defn}
  A functor $\cM^\opp\to \sSet$ is \emph{small} if it is the left Kan extension of a functor $\cJ^\opp\to \sSet$, where $\cJ$ is a small full subcategory of $\cM$. We denote by $\cP(\cM)$ the category of small functors from $\cM^\opp\to \sSet$. 
  \end{defn}
  As stated in \cite{DL07}, small functors can be alternatively characterized as small weighted colimits of representable functors. 
  
  Chorny \cite{Ch14} defined the $\cO$-relative model structure on $\cP(\cM)$, denoted $\cP(\cM,\cO)$. 
  A morphism $X\to Y$ in $\cP(\cM,\cO)$ is a weak equivalence (resp. fibration) if for every $O\in \cO$, the induced map $X(O)\to Y(O)$ is a weak equivalence (resp. fibration) in $\sSet$ with the standard model structure. 
  As discussed in \cite{Ch14}, Proposition 2.8, the $\cO$-relative model structure on $\cP(\cM)$ is the same as the $\{\hom(-,O):O\in \cO\}$-equivariant model structure on $\cP(\cM)$. 
  
  Now we can state Chorny's theorem. 
  \begin{thm}[\cite{Ch14}, Theorem 3.1]\label{Chorny}
  Let $\cM$ be a bicomplete $\sSet$-enriched category and $\cO$ be a full subcategory of orbits in $\cM$. 
  Then there is a Quillen equivalence 
  $$Z:\cP(\cM,\cO)\rls \cM:Y$$
  where $Y$ sends an object $M\in \cM$ to the functor $\hom(M,-)$ corepresented by $M$, and $Z=\Id_\cM\ot_\cM -$ is the coend with the identity functor. 
  \end{thm}
  
  \section{Comparison between $\cO$-equivariant model structures}
  In this section, we prove a comparison result between $\cO$-equivariant model structures on different $\sSet$-enriched categories. 
  
  \begin{thm}\label{Comparison}
  Let $\cM$, $\cN$ be two simplicial categories with an adjunction $$L:\cM\rls \cN:R.$$
  Let $\cO_\cM$ be a collection of orbits in $\cM$ and $\cO_\cN$ be a collections of orbits in $\cN$, in the sense of Definition \ref{OrbitsDefn}.
  Assume that $L \cO_\cM\sse \cO_\cN$ and $R\cO_\cN\sse \cO_\cM$, and that $(L,R)$ restricts to an equivalence of categories between $\cO_\cM$ and $\cO_\cN$. 
  Then $(L,R)$ defines a Quillen equivalence between $\cM$ with the $\cO_\cM$-equivariant model structure and $\cN$ with the $\cO_\cN$-equivariant model structure.
  \end{thm}
  
  We prove the theorem in several steps. Recall that $\cP(\cM)$ is the category of small functors from $\cM^\opp$ to $\sSet$. 
  \begin{lemma}\label{PshAdjoint}
  In the setting of Theorem \ref{Comparison}, we have an adjunction $$R^*:\cP(\cM)\rls \cP(\cN):L^*.$$ where $L^*$ and $R^*$ are restrictions along $L$ and $R$, respectively. 
  \end{lemma}
  
  \begin{proof}
  It is not even obvious that $L^*$ and $R^*$ are well-defined functors. 
  By definition, every object in $\cP(\cM)$ is the left Kan extension along the inclusion functor of a small full subcategory of $\cM$. So we have a well-defined pushforward functor $L_*:\cP(\cM)\to \cP(\cN)$ given by left Kan extension along $L$. 
  We would like to apply \cite{DL07}, Proposition 3.3 to show that $L^*$ is well defined and is the right adjoint of $L_*$. 
  So we need to verify the condition that $\cN(L-,X):\cM^\opp\to \sV$ is small for every object $X\in \cN$. 
  In fact, since $L$ is a left adjoint, $\cN(L-,X)\isom\cM(-,RX)$ is a representable functor, and is small. 
  
  Now we have an adjunction $$L_*:\cP(\cM)\rls \cP(\cN):L^*.$$ We would like to prove that $L_*$ and $R^*$ are naturally equivalent. 
  To show this, consider an object $\Lan_i F$ in $\cP(\cM)$, where $i:\cJ\to \cM$ is the inclusion of a small full subcategory, and $F\in \cP(\cJ)$.
  We have
  \begin{align*}
  L_*\Lan_i F=\Lan_{Li} F=\colim_{Li a\to -} F(a)\isom \colim_{i a\to R-} F(a)=R^* \Lan_i F.
  \end{align*}
  \end{proof}
  
  Recall that the model category $\cP(\cM,\cO_\cM)$ is the category $\cP(\cM)$ equipped with the $\cO_\cM$-relative model structure where a morphism is a weak equivalence (resp. fibration) if and only if it is a weak equivalence (resp. fibration) evaluated at every object in $\cO_\cM$.
  \begin{lemma}\label{PshQuillen}
  In the setting of Theorem \ref{Comparison}, we have a Quillen equivalence $$R^*:\cP(\cM,\cO_\cM)\rls \cP(\cN,\cO_\cN):L^*.$$
  \end{lemma}
  \begin{proof}
  By Lemma \ref{PshAdjoint}, $(R^*,L^*)$ is a pair of adjoint functors. 
  
  We first prove that $(R^*,L^*)$ is a Quillen adjunction, i.e. $L^*$ preserves fibrations and trivial fibrations. 
  A morphism $\eta$ in $\cP(\cN,\cO_\cN)$ is a fibration (resp. trivial fibration) if and only if $\eta(O)$ is a fibration (resp. trivial fibration) for every $O\in \cO_\cN$. 
  Its restriction $L^* \eta$ in $\cP(\cM,\cO_\cM)$ is a fibration (resp. trivial fibration) if and only if $L^* \eta(O^\p)=\eta(LO^\p)$ is a fibration (resp. trivial fibration) for every $O^\p\in \cO_\cM$. We have $L\cO_\cM\sse \cO_\cN$, so if $\eta$ is a fibration (resp. trivial fibration), then $L^*\eta$ is a fibration (resp. trivial fibration). 
  
  Then we prove that $(R^*,L^*)$ is a Quillen equivalence, i.e. when $M\in \cP(\cM,\cO_\cM)$ is a cofibrant object and $N\in \cP(\cN,\cO_\cN)$ is a fibrant object, a morphism $M\to L^*N$ is a weak equivalence if and only if its adjoint $R^*M\to N$ is a weak equivalence.
  The morphism $M\to L^*N$ is a weak equivalence if and only if $M(O)\to L^*N(O)=N(LO)$ is a weak equivalence for every object $O\in \cO_\cM$.
  On the other hand, $R^*M\to N$ is a weak equivalence if and only if $M(RO^\p)=R^*M(O^\p)\to N(O^\p)$ is a weak equivalence for every object $O^\p\in \cO_\cN$.
  Therefore if for every $O\in \cO_\cM$ we can find $O^\p\in \cO_\cN$ such that $RO^\p\isom O$, and for every $O^\p\in \cO_\cN$ we can find $O\in \cO_\cM$ such that $LO\isom O^\p$, the above two conditions coincide.
  This is the case because $(L,R)$ defines an equivalence of categories between $\cO_\cM$ and $\cO_\cN$.
  \end{proof}
  
  \begin{proof}[Proof of Theorem \ref{Comparison}]
  By Chorny's theorem, we have the following diagram.
  \[
   \xymatrix{\cM\ar@<.5ex>[r]^{L}\ar@<.5ex>[d]^{Y} &\cN\ar@<.5ex>[l]^{R}\ar@<.5ex>[d]^{Y}\\
   \cP(\cM,\cO_\cM)\ar@<.5ex>[u]^{Z}\ar@<.5ex>[r]^{R^*}& \cP(\cN,\cO_\cN)\ar@<.5ex>[u]^{Z}\ar@<.5ex>[l]^{L^*}}
  \]
  
  The vertical $(Z,Y)$ pairs are Quillen equivalences by  Theorem \ref{Chorny}, and the pair $(R^*,L^*)$ is a Quillen equivalence by Lemma \ref{PshQuillen}. 
  To show that $(L,R)$ is a Quillen equivalence, we apply the 2-out-of-3 property of Quillen equivalences. 
  So we need to show that $LZ\isom ZR^*$, and $YR\isom L^*Y$. 
  
  Consider any object $M\in \cP(\cM,\cO_\cM)$. 
  To show that $LZM\isom ZR^*M$, we only need to show that they represent the same functor in $\cN$. 
  For any $n\in \cN$, we have 
  \begin{align*}
  \cN(LZM,n) \isom \cM(ZM, Rn) \isom \cP(\cM,\cO_\cM) (M, YRn)
  \end{align*}
  and
  \begin{align*}
  \cN(ZR^*M,n)\isom \cP(\cN,\cO_\cN)(R^*M, Yn) \isom \cP(\cM,\cO_\cM)(M, L^*Yn).
  \end{align*}
  It remains to prove that $YR\isom L^*Y$.
  For any $m\in \cM$, we have
  \begin{align*}
  (YRn)(m)=\cM(m, Rn)\isom \cN(Lm,n)=(L^*Y n)(m).
  \end{align*}
  \end{proof}
  
  \section{The $\cO$-equivariant model structure on $\Cat^I$}
  In this section, we prove Theorem \ref{MainThm} for $\Cat^I$. 
  
  Let $I$ be a small category and $\cO$ be a locally small class of objects in $\Set^I$ such that $\colim_I O=*$ for every $O\in \cO$. 
  Clearly $\Set^I$ embeds naturally in $\sSet^I$, $\Cat^I$, $\Ac^I$ and $\Pos^I$. 
  Farjoun \cite{Far87} proved that $\cO$ is a collection of orbits in $\sSet^I$. Therefore $\sSet^I$ admits the $\cO$-equivariant model structure. 
  
  In order to construct the $\cO$-equivariant model structure on $\Cat^I$, we need $\Cat^I$ to be $\sSet$-enriched.
  We know that $\Cat$ is Carterian closed, i.e. enriched over itself. 
  Let $\ihom$ denote the internal hom of $\Cat$ . 
  The nerve functor $N:\Cat\to \sSet$ gives the usual $\sSet$-enriched structure on $\Cat$, but it is not suitable for our purpose. 
  We consider the $\sSet$-enriched on $\Cat$ given by the strict monoidal functor $\Ex^2 N:\Cat\to \sSet$. 
  Let $\hom$ denote $\Ex^2 N \ihom$. 
  
  Now we show that $\Cat$ is a bicomplete $\sSet$-enriched category. 
  It is well-known that underlying category of $\Cat$ is bicomplete, so we only need to show that the $\sSet$-enriched structure is powered and copowered. 
  For $C\in \cC$ and $X\in \sSet$, the copower $C\ot X$ is $C\t c\Sd^2 X$ and the power $[X,C]$ is $\ihom(c\Sd^2 X,C)$.
  The correctness is easily verified. 
  Therefore $\Cat$ is a bicmoplete $\sSet$-enriched category. 
  
  Slightly abusing notation, we use $\ihom$ (resp. $\hom$) to denote the hom functor in $\Cat^I$ induced from $\ihom$ (resp. $\hom$) in $\Cat$. 
  
  \begin{thm}[Theorem \ref{MainThm} for $\Cat^I$]\label{MainThmCat}
  Let $I$ be a small category and $\cO$ be a locally small class of objects in $\Set^I$ such that $\colim_I O=*$ for every $O\in \cO$. 
  Then $\Cat^I$ admits the $\cO$-equivariant model structure, and there is a Quillen equivalence $$c\Sd^2: \sSet^I \lrs \Cat^I:\Ex^2 N$$ where both sides are equipped with the $\cO$-equivariant model structures. 
  \end{thm}
  
  We prove that $\cO$ is a collection of orbits in $\Cat^I$ by first proving analogous orbit axioms Q1-Q3 with $\hom$ replaced with $\ihom$. 
  
  \begin{prop}[Analogue of Q1]\label{VQ1}
  Let \[
\xymatrix{
O \t K \ar[r] \ar[d]  &  X_{a} \ar[d]\\
O \t L \ar[r] &  X_{a+1}\\
}
\]
be a pushout in $\Cat^I$ where $O\in \cO$ and $K\inj L\in \cI_\Cat$. 
Then \[
\xymatrix{
\ihom(O^\p, O \t K) \ar[r] \ar[d]  &  \ihom(O^\p, X_{a}) \ar[d]\\
\ihom(O^\p, O \t L) \ar[r] &  \ihom(O^\p, X_{a+1})\\
}
\]
is a pushout in $\Cat$ for $O^\p\in \cO$. 
  \end{prop}
  
  The proof is divided into several steps.
  
  \begin{lemma}\label{DwyerProduct}
  Let $K\to L$ be a Dwyer map between posets and $O\in \Cat^I$ be a diagram. Then 
  the natural map $O\t K\to O\t L$ is pointwise a Dwyer map between posets. 
  \end{lemma}
  \begin{proof}
  Assume $i:K\to L$ factors as $K\xrw{f}W \xrw{j} L$, where $f$ is a monomorphism, $j$ is a cosieve, and $f$ admits a right adjoint $r$.
  We prove that $A\t i:A\t K\inj A\t L$ is a Dwyer map between posets for any set $A$. 
  It is clear that $A\t i$ is a sieve. 
  Consider the sequence $A\t K\xrw{A\t f} A\t W\xrw{A\t j} A\t L$. Clearly $A\t j$ is a cosieve.
  It remains to prove that $A\t r:A\t W\to A\t K$ is the right adjoint of $A\t f$. 
  For any $(a,x)\in A\t K$ and $(b,y)\in A\t W$, we have 
  \begin{align*}
  &(A\t W)((A\t f)(a,x),(b,y)) \\
  &= (A\t W)((a,fx),(b,y)) \\
  &= \left \{\begin{array}{ll}
     W(fx,y) &\ \text{if}\ a=b, \\
     \es &\ \text{otherwise}
     \end{array} \right. \\
  &= \left \{\begin{array}{ll}
     K(x,ry) &\ \text{if}\ a=b, \\
     \es &\ \text{otherwise}
     \end{array} \right. \\
  &= (A\t K)((a,x),(b,ry)) \\
  &= (A\t K)((a,x),(A\t r)(b,y)).
  \end{align*}
  \end{proof}
  
  \begin{lemma}\label{HomChange}
  Let $O\in \cO$, and $D$ be a diagram in $\Cat^I$. Let $K$ be a poset.
  Then $\ihom(O,D\t K)=\ihom(O,D)\t K$. 
  \end{lemma}
  \begin{proof}
  We have a natural monomorphism $\ihom(O,D)\t K\to \ihom(O,D\t K)$ where a pair $(f_0,k)\in \ihom(O,D)\t K$ is sent to the morphism $x\mto (f_0(x),k)$.
  Let $f\in \ihom(O,D\t K)$. We prove that for all $i\in I$ and $a\in O(i)$, the second components of $f(a)\in D(i)\t K$ are the same. If this holds, then $f$ is in the image of the natural monomorphism and the lemma follows. 
  
  Consider a morphism $g:i\to j$ in $I$. For all $x\in O(i)$, we have $f(x)\in D(i)\t K$ and the second components of $f(g(x))\in D(j)\t K$ are the same because the maps $D(i)\t K\to D(j)\t K$ in the diagram $D\t K$ preserve the second component. 
  
  By assumption, $\colim_I O=*$. 
  So any two elements in $\bscup_{i\in I} O(i)$ are equivalent the under the equivalence relation generated by $x\sim g(x)$ for every object $i,j\in I$, morphism $g:i\to j$, and $x\in O(i)$. Therefore the images of them under the map $f$ have the same second components. 
  \end{proof}
  
  \begin{lemma}\label{HomDwyer}
  Let $K$, $L$, $O$ be as in Lemma \ref{DwyerProduct} and $O^\p\in \cO$. 
  Then the natural map $\ihom(O^\p, O\t K) \to \ihom(O^\p, O\t L)$ is a Dwyer map between posets. 
  \end{lemma}
  \begin{proof}
  By Lemma \ref{HomChange}, we only need to prove that $\ihom(O^\p,O)\t K\to \ihom(O^\p,O)\t L)$ is a Dwyer map between posets. This follows from Lemma \ref{DwyerProduct}.
  \end{proof}
  
  Before proving Proposition \ref{VQ1}, we need one lemma from \cite{BMOOPY13}, which gives an explicit description of pushouts in $\Cat$ along Dwyer maps between posets. 
  \begin{lemma}[\cite{BMOOPY13}, Lemma 2.5]\label{DwyerPushout}
  Let $i:A\to B$ be a Dwyer map between posets with cosieve $W$ and retraction $r$, and $F:A\to C$ be a functor. Let $D$ be the pushout of $i$ and $F$. Then $\ob(D)=\ob(C)\sqcup (\ob(B)\bl \ob(A))$ and for $d,d^\p\in \ob(D)$, 
  \[ D(d,d^\p)=\left \{ 
  \begin{array} {ll}
  B(d,d^\p) & \text{if } d,d^\p\in \ob(B)\bl \ob(A), \\
  C(d,d^\p) & \text{if } d,d^\p\in \ob(C), \\
  \es & \text{if } d\in \ob(B)\bl \ob(A) \text{ and } d^\p\in \ob(C), \\
  C(d,F(r(d^\p))) & \text{if } d\in \ob(C) \text{ and } d^\p\in \ob(B)\bl \ob(A).
  \end{array}\right.
  \]
  \end{lemma}
  
  \begin{proof}[Proof of Proposition \ref{VQ1}]
  Let $D$ be the following pushout. 
  \[
\xymatrix{
\ihom(O^\p, O \t K) \ar[r] \ar[d]  &  \ihom(O^\p, X_{a}) \ar[d]\\
\ihom(O^\p, O \t L) \ar[r] & D\\
}
\]
  We prove that $D=\ihom(O^\p,X_{a+1})$. 
  By Lemma \ref{DwyerPushout}, \[\ob(D)= \ob(\ihom(O^\p,X_a)) \sqcup (\ob(\ihom(O^\p,O\t L)) \bl \ob(\ihom(O^\p,O\t K))).\]
  By Lemma \ref{HomChange}, we have $\ihom(O^\p,O\t L)=\ihom(O^\p,O)\t L$ and $\ihom(O^\p,O\t K)=\ihom(O^\p,O)\t K$. So 
  \[\ob(D)=\ob(\ihom(O^\p,X_a)) \sqcup \ob (\ihom(O^\p,O)\t (L\bl K)).\]
  
  Now let us consider $X_{a+1}$. By Lemma \ref{DwyerPushout}, for each $i$, \[\ob(X_{a+1}(i))=\ob(X_a(i))\sqcup \ob (O(i)\t (L\bl K)).\]
  
  Clearly there is a monomoprhism $D\to \ihom(O^\p, X_{a+1})$. We prove that this is an isomorphism. 
  
  Let $f\in \ihom(O^\p,X_{a+1})$.
  For any objects $i,j\in I$, morphism $g:i\to j$, and $x\in O^\p(i)$, we have $f(x)\in X_a(i)$ if and only if $f(g(x))\in X_a(j)$. 
  By assumption, $\colim_I\cO^\p=*$. So either $f(x)\in X_a(i)$ for all $i\in I$ and $x\in O^\p(i)$, or $f(g(x))\in O(i)\t (L\bl K)$ for all such $i$ and $x$. 
  So \[\ob(\ihom(O^\p,X_{a+1}))=\ob(\ihom(O^\p,X_a)) \sqcup \ob (\ihom(O^\p,O)\t (L\bl K))=\ob(D).\]
  
  So the natural map $D\to \ihom(O^\p,X_{a+1})$ is an isomorphism on objects. It is easy to see that the map is also an isomorphism on morphisms. 
  \end{proof}
  
  \begin{prop}[Analogue of Q2 and Q3]\label{VQ2}
  Let $O\in \cO$ and $X_1\to \cdots \to X_a\to X_{a+1}\to \cdots$ be a continuous transfinite sequence in $\Cat^I$ where each map $X_a\to X_{a+1}$ is as in Proposition \ref{VQ1}.
  The natural map
  \[ \colim_{a} \ihom(O,X_a)\to \ihom(O,\colim_a X_a)\]
  is an isomorphism.
  \end{prop}
  \begin{proof}
  By Lemma \ref{DwyerPushout}, each $X_a\to X_{a+1}$ is a monomorphism. So $\colim_a X_a$ can be understood as an infinite union of $X_a$, and $\colim_a \ihom(O, X_a)$ can be understood as an infinite union of $\ihom(O, X_a)$.
  
  Let $f\in \ihom(O, \colim_a X_a)$. We prove that there exists some index $b$ such that $f$ factors as $O\to X_b\to \colim_a X_a$. If this holds then $f\in \colim_a \ihom(O,X_a)$ and the proposition follows.
  
  Consider any $i\in I$, an index $b$ and an object $x\in X_b(i)\sse \colim_a X_a(i)$.
  By the explicit description of $X_a\to X_{a+1}$, we see that
  \begin{enumerate}
  \item If $g:i\to j$ is a map in $I$, then $f(g(x))\in X_b(j)$.
  \item If $g:j\to i$ is a map in $I$ and $y\in \colim_a X_a(j)$ is an object such that $g(y)=x$, then $f(y)\in X_b(j)$.
  \end{enumerate}
  
  Therefore for any objects $i,j\in I$, morphism $g:i\to j$ and $x\in O(i)$, we have that $f(x)\in X_b(i)$ if and only if $f(g(x))\in X_b(j)$. 
  By assumption, $\colim_I O=*$. So there exists some index $b$ such that for all $i\in I$ and $x\in O(i)$, we have $f(x)\in X_b(i)$. 
  \end{proof}
  
  Now we transfer our analogous orbit axioms (Proposition \ref{VQ1} and \ref{VQ2}) to the actual orbit axioms. 
  \begin{prop}[Q1]\label{UQ1}
  In the setting of Theorem \ref{MainThmCat}, the class $\cO\sse \Cat^I$ satisfies Q1. 
  \end{prop}
  \begin{proof}
  Let $O,O^\p\in \cO$ and $K\inj L\in \cI_\sSet$. 
  Let $K^\p\inj L^\p=c\Sd^2 (K\inj L) \in \cI_\Cat$. 
  By definition, $O\ot K=O\t K^\p$ and $O\ot L=O\t L^\p$. 
    Let \[
\xymatrix{
O \ot K \ar[r] \ar[d]  &  X_{a} \ar[d]\\
O \ot L \ar[r] &  X_{a+1}\\
}
\]
be a pushout in $\Cat^I$. 
By Proposition \ref{VQ1}, \[
\xymatrix{
\ihom(O^\p, O \ot K) \ar[r] \ar[d]  &  \ihom(O^\p, X_{a}) \ar[d]\\
\ihom(O^\p, O \ot L) \ar[r] &  \ihom(O^\p, X_{a+1})\\
}
\]
is a pushout in $\Cat$. 

  In fact, by the proof of Proposition \ref{VQ1}, this is a pushout along a Dwyer map between posets. By \cite{Tho80}, Proposition 4.3, the natural map 
  $$D=N\ihom(O^\p, X_{a})\cup_{N\ihom(O^\p, O \ot K)} N\ihom(O^\p, O \ot L)\to N\ihom(O^\p, X_{a+1})$$
  is a weak equivalence. 
  
  There is a natural transformation $\eta:\Id\to \Ex$ which is objectwise a weak equivalence (\cite{GJ09}, Theorem III.4.6). 
  So we have a commutative cube 
  $$\begin{tikzpicture}
  \matrix (m) [matrix of math nodes, row sep=1em,
    column sep=0em]{
    & N\ihom(O^\p, O \ot K) & & N\ihom(O^\p, X_{a}) \\
    \hom(O^\p, O \ot K) & & \hom(O^\p, X_{a}) & \\
    & N\ihom(O^\p, O \ot L) & & D \\
    \hom(O^\p, O \ot L) & & \hom(O^\p, X_{a+1}) & \\};
  \path[-stealth]
    (m-1-2) edge (m-1-4) edge (m-2-1)
            edge (m-3-2)
    (m-1-4) edge (m-3-4) edge (m-2-3)
    (m-2-1) edge (m-2-3)
            edge (m-2-3) edge (m-4-1)
    (m-3-2) edge (m-3-4)
            edge (m-4-1)
    (m-4-1) edge (m-4-3)
    (m-3-4) edge (m-4-3)
    (m-2-3) edge (m-4-3)
            edge (m-4-3);
\end{tikzpicture}$$
  where the back square is a homotopy pushout (because it is a pushout and the map $N\ihom(O^\p,O\ot K)\to N\ihom(O^\p,O\ot L)$ is a cofibration), and the four arrows from the back square to the front square are weak equivalences. 
  Hence the front square is a homotopy pushout. 
  \end{proof}
  
  \begin{prop}[Q2 and Q3]\label{UQ2}
  In the setting of Theorem \ref{MainThmCat}, the class $\cO\sse \Cat^I$ satisfies Q3, thus also satisfies Q2. 
  \end{prop}
  \begin{proof}
  The nerve functor $N$ commutes with filtered colimits (\cite{Lac06}). So we only need to prove that the functor $\Ex$ commutes with filtered colimits. 
  This is true because $(\Ex -)_n$ is corepresented by $\Sd \De^n$, which is a compact object. 
  \end{proof}
  
  \begin{prop}\label{LocallySmall}
  In the setting of Theorem \ref{MainThmCat}, the class $\cO$ is locally small in $\Cat^I$.
  \end{prop}
  \begin{proof}
  The inclusion functor $\Set^I\to \Cat^I$ is the left adjoint of the forgetful functor $\Cat^I\to \Set^I$ that forgets the morphisms. 
  So the propositions follows from that $\cO$ is locally small in $\Set^I$.
  \end{proof}
  
  Now we can construct the $\cO$-equivariant model structure on $\Cat^I$. 
  \begin{proof}[Proof of Theorem \ref{MainThmCat}]
  By Proposition \ref{UQ1}, \ref{UQ2}, and \ref{LocallySmall}, the class $\cO$ is locally small in $\Cat^I$ and satisfies Q1-Q3, thus is a collection of orbits.
  The existence of the model structure follows from Farjoun's Proposition \ref{EquivModelStructure}.

  The functors $c$, $\Sd$, $\Ex$, $N$ all preserve sets. 
  So $c \Sd^2$ and $\Ex^2 N$ preserve sets, i.e. they restrict to equivalences of categories between $\cO$ considered as a full subcategory of $\Cat^I$ and $\cO$ considered as a full subcategory of $\sSet^I$. The Quillen equivalence follows from Theorem \ref{Comparison}. 
  \end{proof}
  
  \begin{rmk}\label{RelThomason}
  The $\cO$-equivariant model structure on $\Cat^I$ can also be described as follows. 
  A morphism $X\to Y$ in $\Cat^I$ is a weak equivalence (resp. fibration) if and only if for all $O\in \cO$, 
  $\ihom(O,X)\to \ihom(O,Y)$ is a weak equivalence (resp. fibration) in $\Cat$ with the Thomason model structure.
  
  Note, on the other hand, that our proof does not use the Thomason model structure directly. If we take $I=\mathbf{1}$ and $\cO=\{*\}$ in Theorem \ref{MainThmCat}, then the $\cO$-equivariant model structure on $\Cat^I=I$ reduces to the Thomason model structure. 
  \end{rmk}
  
  \section{The $\cO$-equivariant model structures on $\Ac^I$ and $\Pos^I$}
  In this section we prove Theorem \ref{MainThm} for $\Ac^I$ and $\Pos^I$. 
  
  Recall that $\Ac$ is the category of the small acyclic categories and $\Pos$ is the category of posets.
  Let $\cC$ denote the category $\Ac$ or $\Pos$. 
  The category $\Pos$ is a full subcategory of $\Ac$, which is in turn a full subcategory of $\Cat$. So we can see $\Ex^2 N$ as a functor from $\cC$ to $\sSet$, and define $\hom$ and $\ihom$ in $\cC^I$ using the corresponding hom functors in $\Cat^I$. 
  By \cite{Tho80}, Lemma 5.6, the functor $c \Sd^2$ takes values in posets. 
  Therefore we can see $c\Sd^2$ as a functor from $\sSet$ to $\cC$. 
  
  The functor $\hom$ gives $\cC$ an $\sSet$-enriched structure. 
  It is well known that the underlying category of $\cC$ is bicomplete.
  The power and copower structure of $\cC$ as an $\sSet$-enriched category is similar to that of $\Cat$. For $C\in \cC$ and $X\in \sSet$, the copower $C\ot X$ is $C\t c\Sd^2 X$ and the power $[X,C]$ is $\ihom(c\Sd^2 X,C)$.
  Therefore $\cC$ is a bicomplete $\sSet$-enriched category. 
  

  \begin{thm}[Theorem \ref{MainThm} for $\Ac^I$ and $\Pos^I$] \label{MainThmAcPos}
  Let $\cC$ be $\Ac$ or $\Pos$. 
  Let $I$ be a small category and $\cO$ be a locally small class of objects in $\Set^I$ such that $\colim_I O=*$ for every $O\in \cO$. 
  Then $\cC^I$ admits the $\cO$-equivariant model structure and there is a Quillen equivalence $$c\Sd^2: \sSet^I \lrs \cC^I:\Ex^2 N$$ where both sides are equipped with the $\cO$-equivariant model structures. 
  \end{thm}
  
  We prove that $\cO$ satisfies the orbit axioms Q1-Q3. 
  \begin{prop}[Q1]\label{CQ1}
  In the setting of Theorem \ref{MainThmAcPos}, the class $\cO\sse \cC^I$ satisfies Q1. 
  \end{prop}
  \begin{proof}
  By \cite{Tho80}, Lemma 5.6, the inclusion $\Pos\to \Cat$ preserves pushouts along Dwyer maps between posets. 
  By \cite{Bru15}, Proposition 4.5, the inclusion $\Ac\to \Cat$ preserves pushouts whose leg is a Dwyer map between posets. 
  So the relevant pushouts in $\cC$ can be performed in $\Cat$. 
  \end{proof}
  
  \begin{prop}[Q2 and Q3]\label{CQ2}
  In the setting of Theorem \ref{MainThmAcPos}, the class $\cO\sse \cC^I$ satisfies Q3, thus also satisfies Q2. 
  \end{prop}
  \begin{proof}
  By \cite{Tho80}, Lemma 5.6, the inclusion $\Pos\to \Cat$ preserves filtered colimits. 
  By \cite{Bru15}, Proposition 4.1, the inclusion $\Ac\to \Cat$ preserves filtered colimits. 
  So the relevant colimits in $\cC$ can be performed in $\Cat$. 
  \end{proof}
  
  \begin{prop}\label{CLocallySmall}
  In the setting of Theorem \ref{MainThmAcPos}, the class $\cO$ is locally small in $\cC^I$. 
  \end{prop}
  \begin{proof}
  Same as the proof of Proposition \ref{LocallySmall}. 
  \end{proof}
  
  \begin{proof}[Proof of Theorem \ref{MainThmAcPos}]
  By Proposition \ref{CQ1}, \ref{CQ2}, and \ref{CLocallySmall}, the class $\cO$ is a collection of orbits in $\cC^I$. The existence of the model structure follows from Proposition \ref{EquivModelStructure}. The proof of the Quillen equivalence follows from Theorem \ref{Comparison}. 
  \end{proof}
  
  \begin{rmk}\label{RelThomasonAcPos}
  Similar to Remark \ref{RelThomason}, we can also describe the $\cO$-equivariant model structure on $\cC^I$ using the Thomason model structure. 
  A morphism $X\to Y$ in $\cC^I$ is a weak equivalence (resp. fibration) if and only if for all $O\in \cO$, the induced map $\ihom(O,X)\to \ihom(O,Y)$ is a weak equivalence (resp. fibration) in $\cC$ (or $\Cat$) equipped with the Thomason model structure.
  \end{rmk}
  
  \section{Applications}
  In this section we discuss some applications of our theorem.
  \subsection{Equivariant diagrams}
  Let $G$ be a discrete group acting on a small category $I$. 
  Define $G\rt_a I$ to be the category whose objects are $\ob(I)$, and a morphism $i\to j$ is a pair $(g,\al:gi\to j)$ where $g\in G$ and $\al\in \mor(I)$. 
  The category of $G$-diagrams in a category $\cC$ is defined to be $\cC^{G\rt_a I}$. 
  
  When $\cC$ admits cellular fixed-point functors in the sense of Guillou and May \cite{GM11}, Dotto and Moi \cite{DM16} defined a model structure on $\cC^{G\rt_a I}$, called the ``$G$-projective'' model structure. 
  For every $i\in I$, let $G_i\sse G$ be the group of stabilizers of $i$. 
  This group acts on $i$ via the morphism $G_i\inj G\rt_a I$ defined by $*\mto i$ and $g\mto (g,\id)$. 
  A morphism $X\to Y$ in $\cC^{G\rt_a I}$ is a weak equivalence (resp. fibration), if and only if for every $i\in I$ and every subgroup $H$ of $G_i$, the induced map $X(i)^{H}\to Y(i)^H$ is a weak equivalence (resp. fibration) in $\cC$. 
  
  We give an alternative proof of the existence of the model structure using orbits.
  \begin{prop}\label{GProjEquiv}
  Let $\cC$ be $\sSet$ equipped with the standard model structure, or $\Cat$, $\Ac$ or $\Pos$ equipped with the Thomason model structure. Then $\cC^{G\rt_a I}$ admits the ``$G$-projective'' model structure. In fact, this model structure is the $\cO$-equivariant model structure for a certain set $\cO$ of objects in $\Set^{G\rt_a I}$ such that $\colim_{G\rt_a I} O=*$ for every $O\in \cO$.
  \end{prop}
  \begin{proof}
  The existence of the ``$G$-projective'' model structure follows from the second claim. So we only need to find an appropriate $\cO$. 
  
  Fix $k\in I$ and a subgroup $H\sse G_k$. Consider the diagram $O_{k,H}$ in $\Set^{G\rt_a I}$, where
  \begin{enumerate}
  \item an object $i$ is mapped to $(G\rt_a I)(k,i)/H$, where $h\in H$ acts on the morphism set $(G\rt_a I)(k,i)$ by sending a morphism $(g,\al:gk\to i)$ to $(gh,\al:ghk=gk\to i)$;
  \item a morphism $(g,\al:gi\to j)$ is mapped to the map that sends the orbit of $(g^\p,\al^\p:g^\p k\to i)$ to the orbit of $(g g^\p,\al\circ (g \al^\p):gg^\p k\to g i\to j)$.
  \end{enumerate}
  
  In short, $O_{k,H}=(G\rt_a I)(k,-)/H$. 
  We know that $\colim_{G\rt_a I} (G\rt_a I)(k,-)=*$, so $\colim_{G\rt_a I} O_{k,H}=*$. 
  
  Let $\cO$ be the set of $O_{k,H}$, for all $k\in I$ and subgroup $H\sse G_k$.
  Then $\cO$ is locally small (because it is a set) and consists of diagrams whose colimits over $G\rt_a I$ are one-element sets. 
  So we have an $\cO$-equivariant model structure on $\cC^I$ by Theorem \ref{MainThm}. 
  
  To prove that the $\cO$-equivariant model structure is the ``$G$-projective'' model structure, we only need to show that $O_{k,H}$ corepresents the functor $(-)(k)^H$ in $\cC^{G\rt_a I}$. 
  For any $X\in \cC^{G\rt_a I}$, a map from $(G\rt_a I)(k,-)$ to $X$ is uniquely determined by the image of the identity in $(G\rt_a I)(k,-)$ in $X(k)$, i.e. an object in $X(k)$. 
  So a map from $(G\rt_a I)(k,-)/H$ to $X$ is uniquely determined by an object in $X(k)$ that is invariant under $H$ action. 
  This proves that the $0$-skeletons of $\ihom((G\rt_a I)(k,-),X)$ and $X(k)^H$ agree. 
  The proof for higher skeletons is similar. 
  \end{proof}
  
  \begin{rmk}
  It is well known that $\sSet$ equipped with the standard model structure is ``nice''.
  It has been verified in \cite{BMOOPY13} and in \cite{MSZ16} that $\Cat$ and $\Pos$ equipped with the Thomason model structures are ``nice''. 
  So the ``$G$-projective'' model structure is known to exist in these cases. 
  \end{rmk}
  
  \begin{coro}\label{GDiagram}
  In the setting of Proposition \ref{GProjEquiv}, the ``$G$-projective'' model structures on $\sSet^{G\rt_a I}$, $\Cat^{G\rt_a I}$, $\Ac^{G\rt_a I}$ and $\Pos^{G\rt_a I}$ are Quillen equivalent. 
  \end{coro}
  \begin{proof}
  It follows from Proposition \ref{GProjEquiv} and Theorem \ref{MainThm}. 
  \end{proof}
  
  \begin{rmk}
  Take $\cC=\Cat$, and $I=\mathbf{1}$. Let $G$ act on $I$ trivially. 
  Corollary \ref{GDiagram} gives the main result of \cite{BMOOPY13}, which says that $\Cat^G$ admits the fixed point model structure and this model structure is Quillen equivalent to $\sSet^I$ equipped with the fixed point structure. 
  \end{rmk}
  
  \begin{rmk}
  Take $\cC=\Pos$, and $I=\mathbf{1}$. Let $G$ act on $I$ trivially. 
  Corollary \ref{GDiagram} gives the main result of \cite{MSZ16}, which says that $\Pos^G$ admits the fixed point model structure and this model structure is Quillen equivalent to $\sSet^I$ equipped with the fixed point structure. 
  \end{rmk}
  
  \subsection{The collection of all $I$-orbits}
  Finally, we give an example where $\cO$ is a proper class rather than a set. 
  \begin{prop}
  Let $I$ be a small category, $\cC$ be the category $\Cat$, $\Ac$ or $\Pos$, 
  and $\cO$ be the class of all diagrams in $\Set^I$ whose colimits are one-element sets. 
  Then $\cC^I$ admits the $\cO$-equivariant model structure, and this model structure is Quillen equivalent to $\sSet^I$ equipped with the $\cO$-equivariant model structure. 
  \end{prop}
  \begin{proof}
  To apply Theorem \ref{MainThm}, we only need to verify that $\cO$ is locally small in $\Set^I$. 
  This is well-known by \cite{Far87}. 
  \end{proof}
  
  \appendix

  \section{Properness}
  In this appendix we prove that the $\cO$-equivariant model structures on $\sSet^I$, $\Cat^I$, $\Ac^I$, and $\Pos^I$ are proper.
  
  \begin{prop}\label{OEquivRProper}
  Let $\cM$ be a bicomplete $\sSet$-enriched category and $\cO$ be a collection of orbits in $\cM$. Then the $\cO$-equivariant model structure on $\cM$ is right proper. 
  \end{prop}
  \begin{proof}
  We prove that weak equivalences are preserved by pullbacks along fibrations. 
  Let $D=B\t_{A} C$ be a pullback in $\cM$ where the morphism $B\to A$ is a weak equivalence and the morphism $C\to A$ is a fibration. 
Then for every $O\in \cO$, the induced map $\hom(O,B)\to \hom(O,A)$ is a weak equivalence and $\hom(O,C)\to \hom(O,A)$ is a fibration.
  The functor $\hom(O,-)$ preserves limits, so $\hom(O,D)=\hom(O,B)\t_{\hom(O,A)}\hom(O,C)$.
  Because $\sSet$ with the standard model structure is right proper, the map $\hom(O,D)\to \hom(O,C)$ is a weak equivalence. 
  Hence $D\to C$ is a weak equivalence. 
  \end{proof}
  
  \begin{prop}\label{OEquivProper}
  Let $I$ be a small category, $\cC$ be the category $\sSet$, $\Cat$, $\Ac$ or $\Pos$. Let $\cO$ be a locally small class of objects in $\Set^I$ such that $\colim_I O=*$ for every $O\in \cO$. 
  Then the $\cO$-equivariant model structure on $\cC^I$ is left proper. 
  \end{prop}
  \begin{proof}
  When $\cC=\sSet$, we equip $\cC$ with the standard model structure; in other cases, we equip $\cC$ with the Thomason model structure. Then $\cC$ is proper. (The case $\cC=\sSet$ is well-known; the case $\cC=\Cat$ is by \cite{Cis99}; the case $\cC=\Pos$ is by \cite{Rap10}; the case $\cC=\Ac$ is by \cite{Bru15}.)
  
  Let $\ihom$ denote the internal hom of $\cC$. Then in the $\cO$-equivariant model structure on $\cC^I$, a map $X\to Y$ is a weak equivalence (resp. fibration) if and only if for every $O\in \cO$, the induced map $\ihom(O,X)\to \ihom(O,Y)$ is a weak equivalence (resp. fibration) in $\cC$. 
  (The case $\cC=\sSet$ is by definition; the case $\cC=\Cat$ is by Remark \ref{RelThomason}; the case $\cC=\Ac$ or $\Pos$ is by Remark \ref{RelThomasonAcPos}.)
  
  Now we prove that for every $O\in \cO$, the functor $\ihom(O,-):\cC^I\to \cC$ preserves cofibrations.
  Farjoun's proof \cite{Far87} of the existence of the $\cO$-equivariant model structure says that the class $\cI_{\cC^I}=\{O\ot K\to O\ot L:K\to L\in \cI_{\sSet}, O\in \cO\}$ is a class of generating cofibrations in the $\cO$-equivariant model structure. 
  Let $\cI_{\Cat}=\cI_{\Ac}=\cI_{\Pos}=c\Sd^2 \cI_{\sSet}$. Then $\cI_{\cC^I}$ can be written as $\{O\t K\to O\t L:K\to L\in \cI_\cC, O\in \cO\}$. 
  
  By Lemma \ref{HomChange}, for every $O\t K\to O\t L\in \cI_{\cC^I}$ and every $O^\p\in \cO$, the map $\ihom(O^\p,O\t K)\to \ihom(O^\p,O\t L)$ is the same as $\ihom(O^\p,O)\t K\to \ihom(O^\p,O)\t L$, which is a cofibration. 
  By Proposition \ref{VQ1} and \ref{VQ2}, for every $O\in \cO$, the functor $\ihom(O,-)$ preserves transfinite compositions of pushouts of morphisms in $\cI_{\cC^I}$. 
  (The referred lemma and propositions are for $\Cat$, but the proofs hold for $\sSet$, $\Ac$ and $\Pos$.)
  So $\ihom(O,-)$ preserves cofibrations. 
  By \cite{BMOOPY13}, Lemma 3.1, the functor $\ihom(O,-)$ presreves pushouts along cofibrations. 
  
  Now let $D=B\cup_A C$ be a pushout in $\cC^I$ where the morphism $A\to B$ is a weak equivalence and the morphism $A\to C$ is a cofibration. We prove that the map $C\to D$ is a weak equivalence. 
   For every $O\in \cO$, the induced map $\ihom(O,A)\to \ihom(O,B)$ is a weak equivalene and $\ihom(O,A)\to \ihom(O,C)$ is a cofibration. 
   The functor $\ihom(O,-)$ preserves pushouts of cofibrations, so $\ihom(O,D)=\ihom(O,B)\cup_{\ihom(O,A)}\hom(O,C)$.
   Because $\cC$ is left proper, the natural map $\ihom(O,C)\to \ihom(O,D)$ is a weak equivalence. 
   Hence $C\to D$ is a weak equivalence. 
  
  
  \end{proof}
  
  


  \bibliographystyle{alpha}
  \bibliography{ref}
\end{document}